\documentclass[a4paper]{amsart}
\usepackage[colorlinks,citecolor=blue,urlcolor=black,linkcolor=black]{hyperref}
\usepackage{amssymb,cmap,verbatim}
\usepackage{stmaryrd} 
\usepackage{mathrsfs}

\newtheorem{thm}{Theorem}[section]
\newtheorem{lemma}[thm]{Lemma}

\newtheorem{claim}{Claim}[thm]

\newtheorem{prop}[thm]{Proposition}
\newtheorem{fact}[thm]{Fact}

\newtheorem{conjecture}[thm]{Conjecture}

\newtheorem{notation}[thm]{Notation}
\newtheorem*{thmaa}{Theorem~A}

\theoremstyle{definition}
\newtheorem{defn}[thm]{Definition}

\theoremstyle{remark}
\newtheorem{remark}[thm]{Remark}

\DeclareMathOperator{\cf}{cf}

\DeclareMathOperator{\acc}{acc}

\DeclareMathOperator{\dom}{dom}
\DeclareMathOperator{\rng}{rng}
\DeclareMathOperator{\len}{len}
\DeclareMathOperator{\im}{rng}

\DeclareMathOperator{\tp}{tp}

\DeclareMathOperator{\bs}{bs}

\DeclareMathOperator{\ns}{NS}

\newcommand{\s}{\subseteq}
\newcommand{\br}{\blacktriangleright}

\newcommand{\EF}{\operatorname{EF}}

\renewcommand{\mid}{\mathrel{|}\allowbreak}
\renewcommand{\restriction}{\mathbin\upharpoonright}

\newcommand{\redub}{\hookrightarrow_B}                
\newcommand{\reduc}{\hookrightarrow_c}                

\title{Fodor space in generalized descriptive set theory}

\author{Ido Feldman}
\address{Department of Mathematics, Bar-Ilan University, Ramat-Gan 5290002, Israel.}
\urladdr{https://scholar.google.ca/citations?user=OkVgPx0AAAAJ}

\author{Miguel Moreno}
\address{Department of Mathematics and Statistics, University of Helsinki, Helsinki 00560, Finland.}
\address{Institute of Mathematics, University of Vienna, Vienna 1090, Austria.}
\address{Mathematics Research Centre, Tampere University, Tampere, Finland.}
\urladdr{http://miguelmath.com}

\begin{document}
\begin{abstract}
We study the continuous reducibility of isomorphism relations in the space of regresive functions in $\kappa^\kappa$. We show for inaccessible $\kappa$, that if $\mathcal{T}$ is a theory with less than $\kappa$ non-isomorphic models of size $\kappa$ and $\mathcal T'$ is unstable or superstable non-classifiable, then
the isomorphism of models of $\mathcal{T}$ is continuous reducible to the isomorphism of models of $\mathcal{T}'$.
\end{abstract}


\maketitle

\section{Introduction}

One of the main motivations to developed generalized descriptive set theory has been its connection with model theory.
This connection was studied by Friedman, Hyttinen, and Weinstein (former Kulikov) in \cite{FHK13}. From their work, they conjectured the existence of a generalized Borel-reducibility counterpart of Shelah's main gap theorem.

\begin{conjecture}\label{the_conjecture}
    Let $\mathcal T_1$ be a classifiable theory and $\mathcal T_2$ a non-classifiable theory. Is there a Borel reduction from the isomorphism relation of $\mathcal T_1$ to the isomorphism relation of $\mathcal T_2$?
\end{conjecture}

Hyttinen, Weinstein, and Moreno show the consistency of a positive answer to this conjecture in \cite{HKM}. In \cite{HKM}, \cite{Mor21} the authors gave a positive answer to the conjecture when $\kappa$ is a successor cardinal (under certain cardinal assumptions), and $T_2$ a stable unstable theory. Later on, in \cite{Mor23} the conjecture was proved correct for $\kappa$ a successor cardinal (under certain cardinal assumptions). 

\begin{fact}[Borel reducibility Main Gap, Moreno, Theorem A \cite{Mor23}]\label{Borel_Bain_gap}
    Suppose $\kappa=\lambda^+=2^\lambda$, $2^{\mathfrak{c}}\leq\lambda=\lambda^{\omega_1}$, and $\mathcal T_1$ and $\mathcal T_2$ are countable complete first-order theories in a countable vocabulary. If $\mathcal T_1$ is a classifiable theory and $\mathcal T_2$ is a non-classifiable theory, then $$\cong_{\mathcal T_1}\ \reduc\ \cong_{\mathcal T_2} \textit{ and } \cong_{\mathcal T_2}\ \not\redub\ \cong_{\mathcal T_1}.$$
\end{fact}

The case $\kappa$ an inaccessible cardinal was study by Hyttinen and Moreno in \cite{HM} and  \cite{Mor}. 
In \cite{HM} and \cite{Mor} a positive answer was given for $\mathcal T_2$ a stable theory with OCP or superstable with S-DOP. In \cite{MM}, Mangraviti and Motto Ros showed that the Friemand-Hyttinen-Kulikov conjecture holds for classifiable shallow theories with at most $\kappa$ non-isomorphic models.

\begin{fact}[Mangraviti-Motto Ros, Proposition 6.7 \cite{MM}]\label{Man-Mot}
    If $\mathcal T_1$ is theory with at most $\kappa$ non-isomorphic models of size $\kappa$ and $\mathcal T_2$ a non-classfiable thoery, then $\cong_{\mathcal T_1}\ \redub\ \cong_{\mathcal T_2}$.
\end{fact}

From Shelah's Main Gap Theorem 6.1 \cite{Sh90}, we know that under the assumption 

$$\kappa=\aleph_\gamma \textit{ is such that }\beth_{\omega_1}(\mid\gamma\mid)\leq\kappa$$

any classifiable shallow theory $\mathcal T$ has less than $\kappa$ non-isomorphic models of size $\kappa$, and only classifiable shallow theories have less than $2^\kappa$ non-isomorphic models of size $\kappa$. Thus under that assumption, Fact \ref{Man-Mot} tells us that if $\mathcal T_1$ is classifiable and $\mathcal T_2$ is not classifiable, then $\cong_{\mathcal T_1}\ \redub\ \cong_{\mathcal T_2}$.
As mentioned in \cite{MM}, Fact \ref{Man-Mot} is a direct consequence of the number of equivalence classes and the fact that $\cong_{\mathcal  T}$ is $\kappa$-Borel when $\mathcal T$ is a classifiable shallow theory. This is why it is a Borel reduction and not a continuous reduction (which is stronger). A continuous reduction version of Fact \ref{Man-Mot} would be more informative, but it requires a different approach. Clearly Fact \ref{Borel_Bain_gap} gives us a continuous reduction version of Fact \ref{Man-Mot} when $\kappa$ is a successor cardinal. In this paper we will show a continuous reduction version of Fact \ref{Man-Mot} for $\kappa$ strongly inaccessible.

\begin{thmaa}
    Let $\kappa$ be a strongly inaccessible cardinal. If $\mathcal T_1$ is a theory with less than $\kappa$ non-isomorphic models of size $\kappa$ and $\mathcal T_2$ is unstable or a superstable non-classifiable theory, then $\cong_{\mathcal T_1}\ \reduc\ \cong_{\mathcal T_2}$.
\end{thmaa}

Notice that strictly stable (stable unsuperstable) theories are not covered in the previous theorem. As it was showed in \cite{HM} OCP implies that the theory is unsuperstable, thus stable theories with OCP are strictly stable. Unfortunately not all strictly stable theories have the OCP. In a private correspondence with the second author; Mennuni pointed out that in a non-multi-dimensional theory, no type is orthogonal to any set (see Lemma V 5.4 \cite{Sh90}). Thus a strictly stable non-multi-dimensional theory doesn't have the OCP. On the other hand any theory of modules is non-multi-dimensional (see Corollary 3.5 \cite{PP} or Corollary 6.21 \cite{Pre}). So any strictly stable theory of modules (e.g. the theory of $\mathbb Z^{\aleph_0}$ with the Abelian group structure) doesn't have the OCP (this answers Question 4.5 \cite{Mort}, originally asked by Villaveces).

\subsection{Background}

Let us recall the basic definitions, we refer the reader to \cite{FHK13} and \cite{Mor23} for
more information.

The generalized Baire space is the set $\kappa^\kappa$ endowed with the bounded topology, in this topology the basic open sets are of the form $$[\zeta]=\{\eta\in \kappa^\kappa\mid \zeta\subseteq \eta\}$$
where $\zeta\in \kappa^{<\kappa}$.

Let $E_1$ and $E_2$ be
equivalence relations on $\kappa^\kappa$. We say that $E_1$ is
\emph{Borel reducible to $E_2$} if there is a Borel function $f\colon
\kappa^\kappa\rightarrow \kappa^\kappa$ that satisfies $$(\eta,\xi)\in E_1\iff
(f(\eta),f(\xi))\in E_2.$$  We call $f$ a reduction of $E_1$ to
$E_2$ and we denote this by $E_1\redub E_2$. 
 In the case $f$ is continuous,
 we say that $E_1$ is continuously reducible to $E_2$ and
we denote it by $E_1\reduc E_2$.

\begin{defn}\label{struct}
Let $\mathcal{L}=\{Q_m\mid m\in\omega\}$ be a countable relational language.
Fix $\pi$ a bijection between $\kappa^{<\omega}$ and $\kappa$. For every $\eta\in \kappa^\kappa$ define the structure $\mathcal{A}_\eta$ with domain $\kappa$ as follows.
For every tuple $(a_1,a_2,\ldots , a_n)$ in $\kappa^n$ $$(a_1,a_2,\ldots , a_n)\in Q_m^{\mathcal{A}_\eta}\Leftrightarrow Q_m \text{ has arity } n \text{ and }\eta(\pi(m,a_1,a_2,\ldots,a_n))>0.$$
\end{defn}

\begin{defn}
Assume $\mathcal T$ a first-order theory in a relational countable language, we define the isomorphism relation, $\cong_\mathcal  T~\subseteq \kappa^\kappa\times \kappa^\kappa$, as the relation $$\{(\eta,\xi)|(\mathcal{A}_\eta\models \mathcal T, \mathcal{A}_\xi\models \mathcal T, \mathcal{A}_\eta\cong \mathcal{A}_\xi)\text{ or } (\mathcal{A}_\eta\not\models \mathcal T, \mathcal{A}_\xi\not\models \mathcal T)\}$$
\end{defn}

Throughout this paper all the theories will be first-order theories.
We will follow the approach used in \cite{Mor21} and \cite{Mor23} to prove Theorem A. 
The proof of Fact \ref{Borel_Bain_gap} is divided in three parts, it uses the equivalence modulo stationary sets to construct colored ordered trees, then generalized Ehrenfeucht-Mostowski models are constructed from the colored ordered trees, and finally the reduction is constructed from the models.

\begin{defn}\label{clubrel}
Let $S\subseteq \kappa$ a stationary set. Given $\beta\leq\kappa$, we define the equivalence relation $=_S^\beta\ \subseteq\ \beta^\kappa\times \beta^\kappa$, 
as follows $$\eta\mathrel{=^\beta_S}\xi \iff \{\alpha<\kappa\mid \eta(\alpha)\neq\xi(\alpha)\}\cap S \text{ is non-stationary}.$$
\end{defn}

Let $\mu<\kappa$ be a regular cardinal and $S_\mu=\{\alpha<\kappa\mid \cf (\alpha)=\mu\}$ and $S_{\ge\mu}=\{\alpha<\kappa\mid \cf (\alpha)\ge\mu\}$. Let us denote $=_{S_\mu}^\beta$ by $=_{\mu}^\beta$.

The proof of Fact \ref{Borel_Bain_gap} finishes by applying the following reductions from \cite{HKM}.

\begin{fact}[Hyttinen-Kulikov-Moreno, Lemma 2 \cite{HKM}]\label{HKM1}
Let $\lambda<\kappa$ be a regular cardinal. Assume $\mathcal T$ is a countable complete classifiable theory over a countable vocabulary. If $\diamondsuit_\lambda$ holds, then $\cong_\mathcal T \ \reduc\ =^2_\lambda$.
\end{fact}

Notice that the reduction constructed in \cite{Mor23} is a reduction constructed in the generalized Cantor subspace, $2^\kappa$. Therefore the ordered colored trees used have two colors. To prove Theorem A,
We shall need to code more information 
using ordered colored trees.
This means increasing the number of colors to $\kappa$. 
Thus, we construct the reduction in the Fodor space (see below).

Consequently, we will modify the construction of the Ehrenfeucht-Mostowski models, and use a reduction different from Fact \ref{HKM1}.

\subsection{Oganization of this paper}
In Section 2 we introduce the Fodor subspace and the objects relative to it. In Section 3 we construct the ordered colored trees following the ideas of \cite{HK}, \cite{HM}, and \cite{Mor21}. In Section 4 we construct generalized Ehrenfeucht-Mostowski models and study the isomorphism between the them. In Section 5 we prove the main result.

\section{Fodor space}
Following the intuition of  \cite{MM}, if $E$ is a $\kappa$-Borel equivalence relation on $\kappa^\kappa$ with at most $\kappa$ equivalence classes, then the enumeration $\langle x_i\mid x_i\text{ is an $E$- equivalence class}\rangle$ induced the function $\mathcal{F}:\kappa^\kappa\rightarrow\kappa^\kappa$:
 $$\mathcal{F}(\eta)(\alpha)=
   \begin{cases}
    i &\mbox{if } \alpha=0\\
    0 & \mbox{otherwise, }
  \end{cases}
  $$
where $\eta\in x_i$. Clearly $\mathcal{F}$ is a $\kappa$-Borel reduction from $E$ into $0_\kappa$, where $f\ 0_\kappa\ g$ if and only if $f(0)=g(0)$. This kind of reduction obtained from an enumeration of the equivalence classes, may induce a reduction $\mathcal{F}$ on ``\textit{evetually regresive functions}", i.e. for all $\eta\in \kappa^\kappa$ there is $\alpha<\kappa$ such that $\mathcal{F}(\eta)\restriction(\kappa\backslash \alpha)$ is regresive. Thus we are not interested on the whole Baire space but the one of the eventually regresive functions.

\begin{defn}[Fodor space]
    Let $\mathbb F=\{\eta\in \kappa^\kappa\mid \exists\alpha<\kappa(\forall\beta>\alpha\ \eta(\beta)<\beta)\}$. The Fodor space is the set $\mathbb F$ with the induced topology.
\end{defn}
Notice that the subspaces $\beta^\kappa$ are subpaces of $\mathbb F$, in particular the generalized Cantor space.

Let $S\subseteq \kappa$ a stationary set, we define the equivalence relation $=_S^\mathbb F$ as $=_S^\kappa\cap\ (\mathbb F\times \mathbb F)$.
Clearly Fact \ref{HKM1} induced a result on $=_S^\mathbb F$, i.e. if $\diamondsuit_\lambda$ holds, then $\cong_\mathcal T \ \reduc\ =^\mathbb F_\lambda$. In the case of theories $\mathcal T$ with less than $\kappa$ non-isomorphic models, $\diamondsuit_\lambda$ is not needed. This follows from an observation of the proof of Theorem 2.8 \cite{HM}.

\begin{prop}\label{first_reduction}
    Let $S\subseteq \kappa$ be a stationary set. Assume $\mathcal T$ is a countable complete classifiable theory over a countable vocabulary. Then $\cong_\mathcal T \ \reduc\ =^\kappa_S$. 
    
    If $\cong_\mathcal T$ has less than $\kappa$ equivalence classes, then $\cong_\mathcal T \ \reduc\ =^\mathbb F_S$.
\end{prop}

The idea behind the proof is to find an $S$-approximation of $\cong_\mathcal T$ and use the Approximation Lemma (Lemma 6.6 \cite{MP25}). For more on the approximation lemma see \cite{MP25}.

\begin{defn}[The EF game] 
      For every $\alpha\leq\kappa$ the game $\EF^\alpha_\omega(\mathcal{A}\restriction_\alpha,\mathcal{B}\restriction_\alpha)$
on the restrictions 
$\mathcal{A}\restriction\alpha$ and $\mathcal{B}\restriction\alpha$
of the structures $\mathcal{A}$ and $\mathcal{B}$ with domain $\kappa$ is
defined as follows:

In the $n$-th move, first ${\bf I}$ chooses an
ordinal $\beta_n<\alpha$ such that $X_{\beta_n}\subset \alpha$ and
$X_{\beta_{n-1}}\subseteq X_{\beta_n}$. Then ${\bf II}$ chooses an
ordinal $\theta_n<\alpha$ such that
$\dom(f_{\theta_n}),\im(f_{\theta_n})\subset \alpha$,
$X_{\beta_n}\subseteq \dom(f_{\theta_n})\cap\im(f_{\theta_n})$ and
$f_{\theta_{n-1}}\subseteq f_{\theta_n}$ (if $n=0$ then
$X_{\beta_{n-1}}=\emptyset$ and $f_{\theta_{n-1}}=\emptyset$). 

The game ends after $\omega$ moves. Player ${\bf II}$ wins if
$\bigcup_{i<\omega}f_{\theta_i}\colon A\restriction_\alpha\rightarrow
B\restriction_\alpha$ is a partial isomorphism. Otherwise player
${\bf I}$ wins. 

If $\alpha=\kappa$ then this is the same as the
standard $\EF$-game which is usually denoted by $\EF^\kappa_\omega$.
 
  We will write ${\bf I}\uparrow \EF^\alpha_\omega(\mathcal{A}\restriction_\alpha,\mathcal{B}\restriction_\alpha)$ when ${\bf I}$ has a winning strategy in the game $\EF^\alpha_\omega(\mathcal{A}\restriction_\alpha,\mathcal{B}\restriction_\alpha)$. Similarly for ${\bf II}$.
\end{defn}

\begin{defn}
Assume $\mathcal T$ is a complete first order theory in a countable vocabulary. For every $\alpha\leq \kappa$ and $\eta,\xi\in \kappa^\kappa$, we write $\eta\ R_{EF}^\alpha(\mathcal{T})\ \xi$ if one of the following holds, $\mathcal{A}_\eta\restriction_\alpha\not\models \mathcal  T$ and $\mathcal{A}_\xi\restriction_\alpha\not\models \mathcal T$, or $\mathcal{A}_\eta\restriction_\alpha\models \mathcal T$, $\mathcal{A}_\xi\restriction_\alpha\models \mathcal T$ and $\bf{II}\uparrow$ EF$^\kappa_\omega (\mathcal{A}_\eta\restriction_\alpha,\mathcal{A}_\xi\restriction_\alpha)$.
\end{defn}

\begin{fact}[Hyttinen-Moreno, \cite{HM} Lemma 2.7]\label{clubfina}
For every complete first order theory $\mathcal{T}$ in a countable vocabulary, there are club many $\alpha$ such that $R_{EF}^\alpha(\mathcal{T})$ is an equivalence relation. 
\end{fact}

\begin{fact}[Hyttinen-Moreno, Lemma 2.4 \cite{HM}]
Suppose $\eta,\xi\in\kappa^\kappa$ and $\mathcal{T}$ is a complete first order theory in a countable vocabulary. Then the following hold:
\begin{itemize}
\item $\eta\ R_{EF}^\kappa(\mathcal{T})\ \xi\Longleftrightarrow \exists C\subseteq\kappa\text{ a club, such that }\  \eta\ R_{EF}^\alpha(\mathcal{T})\ \xi$ for all $\alpha\in C$.
\item $\neg (\eta\ R_{EF}^\kappa(\mathcal{T})\ \xi)\Longleftrightarrow \exists C\subseteq\kappa\text{ a club, such that }\  \neg(\eta\ R_{EF}^\alpha(\mathcal{T})\ \xi)$ for all $\alpha\in C$. 
\end{itemize}
\end{fact}

\begin{proof}[Proof of Proposition \ref{first_reduction}]
    Let $\mathcal{T}$ be a countable complete classifiable theory over a countable vocabulary. Since $\mathcal{T}$ is classifiable, $R_{EF}^\kappa(\mathcal{T})$ coincide with $\cong_\mathcal T$. To simplify the notation, lets denote by $E_\alpha$ the relation $R_{EF}^\kappa(\mathcal{T})$. So
\begin{itemize}
\item $\eta\ \cong_\mathcal T\ \xi\Longleftrightarrow \exists C\subseteq\kappa\text{ a club, such that }\  \eta\ E_\alpha\ \xi$ for all $\alpha\in C$.
\item $\neg (\eta\ \cong_\mathcal T\ \xi)\Longleftrightarrow \exists C\subseteq\kappa\text{ a club, such that }\  \neg(\eta\ E_\alpha\ \xi)$ for all $\alpha\in C$. 
\end{itemize}

    Let $C\subseteq \kappa$ be the club from Fact \ref{clubfina}. For all $\alpha\in C$, let $\langle x^\alpha_i\mid 0<i<\kappa\rangle$ be an enumeration of the $E_\alpha$-equivalence classes (this can be done since $\kappa^{<\kappa}=\kappa$). Let us define $F:\kappa^\kappa\rightarrow\kappa^\kappa$ as follows:
        $$F(\eta)(\alpha)=\begin{cases} i &\mbox{if }\alpha\in C\mbox{ and } \eta\in x_i^\alpha,\\
    0 & \mbox{otherwise. }\end{cases}$$  
    It is easy to see that $F$ is a continuous reduction from $\cong_\mathcal T$ to $=_S^\kappa$. 
    
    Finally, if $\cong_\mathcal T$ has less than $\kappa$ equivalence classes, then by Morley's conjecture, there is $\theta<\kappa$ (the smallest) such that for all $\alpha\in C$, $E_\alpha$ has less than $\theta$ equivalence classes. So for all $\eta\in \kappa^\kappa$ and $\beta>\theta$, $F(\eta)(\beta)<\theta<\beta$. Thus $F(\eta)\in \mathbb F$ and $\cong_\mathcal T \ \reduc\ =^\mathbb F_S$.
    In this case, the reduction above can be define such that for all $\eta\in \kappa^\kappa$, $F(\eta)\restriction\theta$ is constant to $0$. Thus, we can construct $F$ such that for all $\eta\in \kappa^\kappa$, $F(\eta)\in\{\zeta\in \mathbb F\mid \forall\alpha>0 (\zeta(\alpha)<\alpha)\}$.
\end{proof}

\section{Trees}

We will construct a variation of Hyttinen-Kulikov's coloured trees, this will allow us to construct models from $\eta\in \mathbb F$.

Let $t$ be a tree, for every $x\in t$ we denote by $ht(x)$ the height of $x$, the order type of $\{y\in t | y\prec x\}$. Define $t_\alpha=\{x\in t|ht(x)=\alpha\}$ and $t_{<\alpha}=\cup_{\beta<\alpha}t_\beta$, denote by $x\restriction \alpha$ the unique $y\in t$ such that $y\in t_\alpha$ and $y\prec x$. If $x,y\in t$ and $\{z\in t|z\prec x\}=\{z\in t|z\prec y\}$, then we say that $x$ and $y$ are $\sim$-related, $x\sim y$, and we denote by $[x]$ the equivalence class of $x$ for $\sim$.
An $\alpha, \beta$-tree is a tree $t$ with the following properties:
\begin{itemize}
\item $|[x]|<\alpha$ for every $x\in t$.
\item All the branches have order type less than $\beta$ in $t$.
\item $t$ has a unique root.
\item If $x,y\in t$, $x$ and $y$ have no immediate predecessors and $x\sim y$, then $x=y$.
\end{itemize}

Given a tree $t$, we say that a sequence $( I_\alpha)_{ \alpha<\kappa}$ is a \textit{filtration} of $t$ if the following hold:
\begin{itemize}
    \item it is an increasing sequence of downwards closed subsets of $t$;
    \item $\bigcup_{\alpha<\kappa}I_\alpha=t$;
    \item if $\rho<\kappa$ is a limit ordinal, then $I_\rho=\bigcup_{\alpha<\rho}I_\alpha$;
    \item for all $\alpha<\kappa$, $|I_\alpha|<\kappa$.
\end{itemize}

\begin{defn}[Coloured tree]\label{D.2.1}
Let $\gamma$ be an ordinal smaller than $\kappa$. A coloured tree is a pair $(t,c)$, where $t$ is a $\kappa^+$, $(\gamma+2)$-tree and $c$ is a map $c:t_\gamma\rightarrow \{0,1\}$ (the color function).
\end{defn}

Let $E\subseteq \cf(\lambda)$ be a stationary set. Given a function $\eta$ from $\alpha\leq \lambda\cdot\lambda$ to $\kappa$, we say that $\eta$ is $\lambda$-cofinal to $E$ if for all $s\leq \alpha$ with cofinality $\lambda$, $sup(rang(\eta\restriction s))\in E$.

For all $f\in 2^\kappa$ and stationary $E\subseteq \cf(\lambda)$, define the tree $(\mathcal R^E_f,r^E_f)$ as, $\mathcal R^E_f$ the set of all strictly increasing functions, $\eta$, from some $\alpha\leq \lambda\cdot\lambda$ to $\kappa$, $\lambda$-cofinal to $E$, and if $\eta$ has domain $\lambda\cdot\lambda$, then $r^E_f(\eta)=f(\sup(\im(\eta)))$.

For every pair of ordinals $\alpha$ and $\varrho$, $\alpha<\varrho<\kappa$ and $i<\lambda\cdot\lambda$ define $$\mathcal R^E(\alpha,\varrho,i)=\bigcup_{i< j\leq \lambda\cdot\lambda}\{\eta:[i,j)\rightarrow[\alpha,\varrho)\mid\eta \text{ strictly increasing }\lambda\text{ -cofinal to }E\}.$$

\begin{defn}\label{constants_trees}
If $\alpha<\varrho<\kappa$ and $\alpha,\varrho,\rho\neq 0$, let $\{{}^EZ^{\alpha,\varrho}_\rho|\rho<\kappa\}$ be an enumeration of all downward closed subtrees of $\mathcal R^E(\alpha,\varrho,i)$ for all $i$, in such a way that each possible coloured tree appears cofinally often in the enumeration. Let ${}^EZ^{0,0}_0$ be the tree $(\mathcal R^E_f,r^E_f)$.
\end{defn}

This enumeration is possible because there are at most $$|\bigcup_{i<\lambda\cdot\lambda}\mathcal{P}(\mathcal R^E(\alpha,\varrho,i))|\leq (\lambda\cdot\lambda)\times\kappa=\kappa$$ downward closed coloured subtrees. Since for all $\alpha<\varrho<\kappa$, $|\mathcal R^E(\alpha,\varrho,i)|<\kappa$ there are at most $\kappa\times \kappa^{<\kappa}=\kappa$ coloured trees.
Denote by $Q({}^EZ^{\alpha,\varrho}_\rho)$ the unique ordinal $i$ such that ${}^EZ^{\alpha,\varrho}_\rho\subset \mathcal R^E(\alpha,\varrho,i)$.

\begin{defn}\label{treeconst}

Define for each $H\in 2^\kappa$ and stationary $E\subseteq \cf(\lambda)$, the coloured tree $(J^E_H,c^E_H)$ by the following construction.

For every $H\in 2^\kappa$ and  stationary $E\subseteq \cf(\lambda)$, define $J^E_H=(J^E_H,c^E_H)$ as the tree of all $\eta: s\rightarrow (\lambda\cdot\lambda)\times \kappa^4$, where $s\leq \lambda\cdot\lambda$, ordered by end extension, and such that the following conditions hold for all $i,j<s$:

Denote by $\eta_i$, $1<i\leq5$, the functions from $s$ to $\kappa$ that satisfies, $$\eta(n)=(\eta_1(n),\eta_2(n),\eta_3(n),\eta_4(n),\eta_5(n)).$$
\begin{enumerate}
\item $\eta\restriction n\in J^E_H$ for all $n<s$.
\item $\eta$ is strictly increasing with respect to the lexicographical order on $\lambda\cdot\lambda\times \kappa^4$.
\item $\eta_1(i)\leq \eta_1(i+1)\leq \eta_1(i)+1$.
\item $\eta_1(i)=0$ implies $\eta_2(i)=\eta_3(i)=\eta_4(i)=0$.
\item $\eta_2(i)\ge\eta_3(i)$ implies $\eta_2(i)=0$.
\item $\eta_1(i)<\eta_1(i+1)$ implies $\eta_2(i+1)\ge \eta_3(i)+\eta_4(i)$.
\item For every limit ordinal $\alpha$, $\eta_k(\alpha)=sup_{\iota<\alpha}\{\eta_k(\iota)\}$ for $k\in \{1,2\}$.
\item If $s<\lambda\cdot\lambda$ and $\cf(s)=\lambda$, then $sup(rang(\eta_5))\in E$.
\item $\eta_1(i)=\eta_1 (j)$ implies $\eta_k (i)=\eta_k (j)$ for $k\in \{2,3,4\}$.
\item If for some $k<\lambda\cdot\lambda$, $[i,j)=\eta_1^{-1}\{k\}$, then $$\eta_5\restriction {[i,j)}\in {}^EZ^{\eta_2(i),\eta_3(i)}_{\eta_4(i)}.$$
\noindent Note that 9 implies ${}^EZ^{\eta_2(i),\eta_3(i)}_{\eta_4(i)}\subset \mathcal R^E(\alpha,\varrho,i)$ 
\item If $s=\lambda\cdot\lambda$, then either 
\begin{itemize}
\item [(a)] there exists an ordinal number $m$ such that for every $k<m$, $\eta_1(k)<\eta_1(m)$, for every $k' \ge m$, $\eta_1(k)=\eta_1(m)$, and the color of $\eta$ is determined by ${}^EZ^{\eta_2(m),\eta_3(m)}_{\eta_4(m)}$: $$c^E_H(\eta):=c(\eta_5\restriction {[m,\lambda\cdot\lambda)})$$ where $c$ is the colouring function of $^EZ^{\eta_2(m),\eta_3(m)}_{\eta_4(m)}$;

\end{itemize}

or
\begin{itemize}
\item [(b)] there is no such ordinal $m$ and then $c^E_H(\eta):=H(\sup(\im(\eta_5)))$.
\end{itemize}
\end{enumerate}
\end{defn}

\begin{remark}\label{connection to sup of eta5}
    Since $\eta_5$ is increasing and $\sup(\rng(\eta_3))\ge \sup(\rng(\eta_5))\ge \sup(\rng(\eta_2))$, $\sup(\rng(\eta_2))\ge \sup(\rng(\eta_3))$ and $\sup(\rng(\eta_2))\ge \sup(\rng(\eta_4))$, this leads us to 
$$
\sup(\rng(\eta_4))\leq \sup(\rng(\eta_3))=\sup(\rng(\eta_5))=\sup(\rng(\eta_2)).
$$
Thus, if $\rng(\eta_1)=\lambda\cdot\lambda$ and for $\delta<\kappa$ and $k<5$ are
    such that $\sup(\rng(\eta_k))=\delta$, 
    then also $\sup(\rng(\eta_5))=\delta$.
\end{remark}

Let $S_0,S_1\s S^{\kappa}_{\lambda}$ be a partition of $S^{\kappa}_{\lambda}$ into two disjoint stationary sets, and let $S_1=\uplus_{i<\kappa} E_{i}$ be a partition of $S_1$ into stationary sets. For all $f:\kappa\rightarrow \kappa$ and $\sigma<\kappa$ define 
$$S^f_{\sigma}:=\begin{cases}f^{-1}\{\sigma\}\cap S_0&\text{if }f^{-1}\{\sigma\}\cap S_0\text{ is stationary };\\
E_{\sigma}&\text{Otherwise }.\end{cases}$$

\begin{fact}\label{equality_lemma}
    If $f,g$ are such that $f=^\kappa_{S_0}g$, then for all $\sigma<\kappa$, $S^f_\sigma\neq E_\sigma$ if and only if $S^g_\sigma\neq E_\sigma$.
\end{fact}
\begin{proof}
    Assume, towards a contradiction, that there are $f$ and $g$ such that $f=_{S_0}^\kappa g$, and there is $\sigma\in S_0$ such that $S^f_\sigma= E_\sigma$ and $S^g_\sigma\neq E_\sigma$. From the construction, $S_\sigma^g=g^{-1}\{\sigma\}\cap S_0$ is stationary. Thus there is a club $C\subseteq \kappa$ such that for all $\alpha\in C\cap S_0$, $f(\alpha)=g(\alpha)$. Therefore, $g^{-1}\{\sigma\}\cap S_0\cap C$ is stationary and for all $\alpha\in g^{-1}\{\sigma\}\cap S_0\cap C$, $f(\alpha)=g(\alpha)=\sigma$. We conclude that $f^{-1}\{\sigma\}\cap S_0$ is stationary, so $f^{-1}\{\sigma\}\cap S_0=S_\sigma^f =E_\sigma$, a contradiction.
\end{proof}

For all $f:\kappa\rightarrow \kappa$ and $\sigma<\kappa$ we define $(I^f_\sigma,c^f_\sigma)$ as the coloured tree $(J^{E_\sigma}_H,c^{E_\sigma}_H)$,
where $H\in 2^\kappa$ is the function that satisfies $$H(\alpha)=1\Leftrightarrow \alpha\in S^f_\sigma.$$  
We will write $E$ when it is clear that it is $E_\sigma$.

\begin{notation}
    For all $\sigma<\kappa$, denote by $\mathcal B(I^f_{\sigma})$ the set of leaves of $I^f_{\sigma}$.
    In addition,
    for all $\alpha<\kappa$ let us define $(I^f_\sigma)^\alpha=\{\eta\in I^f_\sigma| rang(\eta)\subset (\lambda\cdot\lambda)\times(\iota)^4\text{ for some }\iota<\alpha\}$.
\end{notation}

The following fact can be proved following the proof of \cite{HK} Claim 2.6 or \cite{HM} Claim 4.8, no changes are needed.

\begin{fact}\label{4.5.1}
Suppose $\xi\in (I_\sigma^f)^\alpha$ and $\eta\in I_\sigma^f$. If $\dom(\xi)$ is a successor ordinal smaller than $\lambda\cdot\lambda$, $\xi\subsetneq \eta$ and for every $k$ in $\dom(\eta)\backslash \dom(\xi)$, $\eta_1(k)=\xi_1(max(\dom(\xi)))$ and $\eta_1(k)>0$, then $\eta\in (I_\sigma^f)^\alpha$.
\end{fact}

\begin{lemma}
    If $f,g$ are such that $f=_{S_0}^\kappa g$, then for all $\sigma<\kappa$, $I^f_{\sigma}\cong I^g_{\sigma}$. 
\end{lemma}
\begin{proof}
    Suppose that $f,g$ are such that $f=^\kappa_{S_0}g$. Let $C$ be a club such that $C\cap S_0\subseteq\{\alpha<\kappa\mid f(\alpha)=g(\alpha)\}$. 
    If $\sigma<\kappa$ is such that $S^f_\sigma=E_\sigma$, then by Lemma \ref{equality_lemma}, $S^g_\sigma=E_\sigma$ and $I^f_{\sigma}\cong I^g_{\sigma}$. We are missing the case when $\sigma<\kappa$ is such that $S^f_\sigma,S^g_\sigma\neq E_\sigma$. Since $S^f_\sigma,S^g_\sigma\neq E_\sigma$, $S^f_\sigma$ and $S^g_\sigma$ are stationary. Thus $f$ and $g$ are equivalent modulo $S^f_\sigma$ and $S^g_\sigma$. The proof follows as in \cite{HK} and \cite{Mor23}.

\end{proof}

\begin{defn}\label{Shtree}
Let $\gamma$ be an ordinal smaller than $\kappa$. Let $K_{tr}^\gamma$ be the class of models $(A,\prec, (P_n)_{n\leq \gamma},<, \wedge)$, where:
\begin{enumerate}
\item There is a linear order $(\mathcal{I},<_\mathcal{I})$ such that $A\subseteq \mathcal{I}^{\leq\gamma}$.
\item $A$ is closed under initial segment.
\item $\prec$ is the initial segment relation.
\item $\wedge(\eta,\xi)$ is the maximal common initial segment of $\eta$ and $\xi$.
\item Let $lg(\eta)$ be the length of $\eta$ (i.e. the domain of $\eta$) and $P_n=\{\eta\in A\mid lg(\eta)=n\}$ for $n\leq \gamma$.
\item For every $\eta\in A$ with $lg(\eta)<\gamma$, define $Suc_A(\eta)$ as $\{\xi\in A\mid \eta\prec \xi\ \&\ lg(\xi)=lg(\eta)+1\}$. $<$ is $\bigcup_{\eta\in A}(<\restriction Suc_A(\eta))$, i.e. if $\xi<\zeta$, then there is $\eta\in A$ such that $\xi,\zeta\in Suc_A(\eta)$.
\item For all $\eta\in A\backslash P_\gamma$, $<\restriction Suc_A(\eta)$ is the induced linear order from $\mathcal{I}$, i.e. $$\eta^\frown \langle x\rangle < \eta^\frown \langle y\rangle \Leftrightarrow x<_\mathcal{I} y.$$
\item If $\eta$ and $\xi$ have no immediate predecessor  and $\{\zeta\in A\mid \zeta\prec\eta\}=\{\zeta\in A\mid \zeta\prec\xi\}$, then $\eta=\xi$.
\end{enumerate}
\end{defn}

For all $f:\kappa\rightarrow \kappa$ and $\sigma<\kappa$, we will use $I^f_\sigma$ to construct $A^f_\sigma$, an element of $K_{tr}^{\lambda\cdot\lambda}$.
Let us denote by $\acc(\kappa)=\{\alpha<\kappa\mid \alpha=0~ \text{or}~ \alpha \text{ is a limit ordinal}\}$. For all $\alpha\in \acc(\kappa)$ and  $\eta\in (I_\sigma^f)^\alpha$ with $\dom(\eta)=m<\lambda\cdot\lambda$ define 
$$W_\eta^\alpha=\{\zeta\mid \dom(\zeta)=[m,s), m\leq s\leq \lambda\cdot\lambda, \eta^\frown\zeta\in (I_\sigma^f)^{\alpha+\omega}, \eta^\frown \langle m,\zeta(m)\rangle\notin (I_\sigma^f)^\alpha\}.$$ 
Notice that by the way $I_\sigma^f$ was constructed, for every $\eta\in I_\sigma^f$ with domain smaller than $\lambda\cdot\lambda$ and $\alpha<\kappa$, the set $$\{(\vartheta_1,\vartheta_2,\vartheta_3,\vartheta_4,\vartheta_5)\in (\lambda\cdot\lambda\times\kappa^4)\backslash (\lambda\cdot\lambda\times\alpha^4)\mid \eta^\frown (\vartheta_1,\vartheta_2,\vartheta_3,\vartheta_4,\vartheta_5)\in (I_\sigma^f)^{\alpha+\omega}\}$$ is either empty or has size $\omega$.
Let $\varsigma_\eta^\alpha$ be an enumeration of this set, when this set is not empty.

Let us denote by $\mathbb{T}=(\kappa\times \omega\times \acc(\kappa)\times (\lambda\cdot\lambda)\times
\kappa\times\kappa\times
\kappa\times\kappa)^{\leq\lambda\cdot\lambda}$. 
For every $\xi \in \mathbb{T}$ there are functions $\{\xi_i\in \kappa^{\leq \lambda\cdot\lambda}\mid 0<i\leq 8\}$ such that for all $i\leq 8$, $\dom(\xi_i)=\dom(\xi)$ and for all $n\in \dom(\xi)$, $\xi(n)=(\xi_1(n),\xi_2(n),\xi_3(n),\xi_4(n),\xi_5(n),\xi_6(n),\xi_7(n),\xi_8(n))$. 
For every $\xi\in \mathbb{T}$ let us denote $(\xi_4,\xi_5,\xi_6,\xi_7,\xi_8)$ by $\overline{\xi}$.

\begin{defn}\label{Gammas}
For all $\alpha\in \acc(\kappa)$ and  $\eta\in \mathbb{T}$ with $\overline{\eta}\in I_\sigma^f$
, $\dom(\eta)=m<\lambda\cdot\lambda$ define
$\Gamma_\eta^\alpha$ as follows:

If $\overline{\eta}\in (I_\sigma^f)^\alpha$, then $\Gamma_\eta^\alpha$ is the set of elements $\xi$ of $\mathbb{T}$ such that:

\begin{enumerate}
\item $\xi\restriction m=\eta,$
\item $\overline{\xi}\restriction \dom(\xi)\backslash m \in W^\alpha_\eta,$
\item $\xi_3$ is constant on $\dom(\xi)\backslash m,$
\item $\xi_3(m)=\alpha$,
\item for all $n\in \dom(\xi)\backslash m$, let $\xi_2(n)$ be the unique $r<\omega$ such that $\varsigma_{\zeta}^\alpha(r)=\overline{\xi}(n)$, where $\zeta=\overline{\xi}\restriction n$.
\end{enumerate}

If $\overline{\eta}\notin (I_\sigma^f)^\alpha$, then $\Gamma_\eta^\alpha=\emptyset$.

\end{defn}

For $\eta\in \mathbb{T}$ with $\overline{\eta}\in I_\sigma^f$, $\dom(\eta)=m<\lambda\cdot\lambda$ define $$\Gamma(\eta)=\bigcup_{\alpha\in \acc(\kappa)}\Gamma_\eta^\alpha .$$
Finally we can define $A^f_\sigma$ by induction. Let $\mathbb I_\sigma^f(0)=\{\emptyset\}$ and for all $n<\lambda\cdot\lambda$, $$\mathbb I_\sigma^f(n+1)=\mathbb I_\sigma^f(n)\cup\bigcup_{\eta\in \mathbb I_\sigma^f(n)~\dom(\eta)=n}\Gamma(\eta).$$ For $n\leq \lambda\cdot\lambda$ a limit ordinal, 
$$\bar{\mathbb I}_\sigma^f(n)=\bigcup_{m<n}\mathbb I_\sigma^f(m)$$
and
$$\mathbb I_\sigma^f(n)=\bar{\mathbb I}_\sigma^f(n)\cup \{\eta\in \mathbb{T}\mid \dom(\eta)=n\ \&\ \forall m<n\ (\eta\restriction m\in \bar{\mathbb I}_\sigma^f(n))\ \&\ \overline{\eta}\in I_\sigma^f\}.$$

For  $0<i\leq 8$ let us denote by $s_i(\eta)=sup\{\eta_i(n)\mid n<\lambda\cdot\lambda\}$ and $s_{\lambda\cdot\lambda}(\eta)=max\{s_i(\eta)\mid i\leq 8\}$. Finally $$A^f_\sigma=\mathbb I_\sigma^f(\lambda\cdot\lambda).$$ 

Define the color function $d_\sigma^f$ by 
\begin{equation*}
d_\sigma^f(\eta)=\begin{cases} c_\sigma^f(\overline{\eta}) &\mbox{if } s_1(\eta)< s_{\lambda\cdot\lambda}(\eta)\\
H(s_1(\eta)) & \mbox{if } s_1(\eta)= s_{\lambda\cdot\lambda}(\eta). \end{cases}
\end{equation*}
Recall that $H$ is the function used to construct $I^f_\sigma$, i.e. $(I^f_\sigma,c^f_\sigma)=(J^E_H,c^E_H)$.
It is clear that $A^f_\sigma$ is closed under initial segments, indeed the relations $\prec$, $(P_n)_{n\leq\lambda\cdot\lambda}$, and $\wedge$ of Definition \ref{Shtree} have a canonical interpretation in $A^f_\sigma$. 

We are missing to define $<\restriction Suc_{A^f_\sigma}(\eta)$ for all $\eta\in A^f_\sigma$ with domain smaller than $\lambda\cdot\lambda$. 
From \cite{Mor23} Remark 3.15 and Theorem 3.16, for all $\varepsilon<\kappa$, there is a linear order $\mathcal I$ such that:
\begin{itemize}
    \item $\mathcal I$ is $\varepsilon$-dense, $(<\kappa)$-stable, and $(\kappa, \varepsilon)$-nice (see below). 
    \item For all $\eta\in A^f_\sigma$, $<\restriction Suc_{A^f_\sigma}(\eta)$ is isomorphic to $\mathcal I$.
    \item If $f,g$ and $\sigma$ are such that $I^f_\sigma\cong I^g_\sigma$, then for all $\sigma<\kappa$, $A^f_{\sigma}\cong A^g_{\sigma}$. 
\end{itemize}

\begin{defn}
Let $\varepsilon<\kappa$ be a regular cardinal, $A$ be a linear order of size $\kappa$ and $\langle A_\alpha\mid\alpha<\kappa\rangle$ a filtration. Then $A$ is \textit{$(\kappa, \varepsilon)$-nice} if
there is a club $C\subseteq \kappa$, such that for all limit $\delta\in C$ with $cf(\delta)\ge\varepsilon$, for all $x\in A$ there is $\beta<\delta$ such that one of the following holds:
\begin{itemize}
\item $\forall\sigma\in A_\delta [\sigma\ge x \Rightarrow \exists\sigma'\in A_\beta\ (\sigma\ge \sigma'\ge x)]$
\item $\forall\sigma\in A_\delta [\sigma\leq x \Rightarrow \exists\sigma'\in A_\beta\ (\sigma\leq \sigma'\leq x)]$
\end{itemize}

\end{defn}

For any $\mathcal{L}$-structure $\mathcal{A}$ we denote by \textit{bs} the set of basic formulas of $\mathcal{L}$ (atomic formulas and negation of atomic formulas). For all $\mathcal{L}$-structure $\mathcal{A}$, $a\in \mathcal{A}$, and $B\subseteq \mathcal{A}$ we define $$tp_{bs}(a,B,\mathcal{A})=\{\varphi(x,b)\mid \mathcal{A}\models \varphi(a,b),\varphi\in bs,b\in B\}.$$
Similarly we define $tp_{at}(a,B,\mathcal{A})$ for atomic formulas.

\begin{defn}
A linear order $A$ is \textit{$(<\kappa)$-stable} if for every $B\subseteq A$ of size smaller than $\kappa$, $$\kappa>|\{tp_{bs}(a,B,A)\mid a\in A\}|.$$

\end{defn}

\begin{defn}
Let $I$ be a linear order of size $\kappa$ and $\varepsilon$ a regular cardinal smaller than $\kappa$. We say that $I$ is \textit{$\varepsilon$-dense} if the following holds.

\textit{If $A,B\subseteq I$} are subsets of size less than $\varepsilon$ such that for all $a\in A$ and $b\in B$, $a<b$, then there is $c\in I$, such that for all $a\in A$ and $b\in B$, $a<c<b$.

\end{defn}

Let us define the following filtration:
$$
(A^f_\sigma)^\alpha=\{\eta\in A^f_\sigma\mid \im(\eta)\subseteq \vartheta\times\omega\times\vartheta\times(\lambda\cdot\lambda)\times\vartheta^4\text{ for some }\vartheta<\alpha\}.
$$

Notice that for all $f,g\in \kappa^\kappa$ and $\alpha,\sigma<\kappa$, $f\restriction\alpha=g\restriction\alpha$ if and only if $(I^f_\sigma)^\alpha=(I^g_\sigma)^\alpha$. Thus, $f\restriction\alpha=g\restriction\alpha$ if and only if $(A^f_\sigma)^\alpha=(A^g_\sigma)^\alpha$.

Let us construct the tree $T^f$.
The construction shall use the following definition.
\begin{defn}
    Let $\langle (T_\sigma,\prec_{\sigma})\mid\sigma<\kappa\rangle$ be a sequence of trees of height $\alpha<\kappa$. \emph{ The disjoint union tree} denoted by $\bigvee_{\sigma<\kappa} T_{\sigma}:=((\biguplus_{\sigma<\kappa} T_{\sigma}\times\{\sigma\})\cup\{\emptyset\},\prec)$
    where $t\prec s$ if, and only if either $t=\emptyset$ or there exists $\sigma<\kappa$ such that 
    $t,s\in T_{\sigma}$ and $t\prec_{\sigma} s$. 

    Briefly, we shall define for each $\sigma<\kappa$ a $\kappa^+$, $(\lambda\cdot\lambda)+2$ tree $T^f_{\sigma}$ which will witness the fact $f^{-1}\{\sigma\}$ is stationary and then take their disjoint union 
    as our colored tree of interest. 
    
\end{defn}

For each $\sigma<\kappa$ let $$T_{\sigma}^f=A^f_\sigma\backslash \{\eta\in A^f_\sigma\mid \dom(\eta)=\lambda\cdot\lambda\ \&\ d^f_\sigma(\eta)=0\}$$ from the previous results, if $f,g$ are such that $f=^\kappa_{S_0}g$, then for all $\sigma<\kappa$, $T^f_{\sigma}\cong T^g_{\sigma}$.
Let $T^f:=\bigvee_{\sigma<\kappa} T^f_{\sigma}$. Clearly if $f,g$ are such that $f=^\kappa_{S_0}g$, then $T^f\cong T^g$.

\begin{notation}\label{notation1}
    For all $\sigma<\kappa$, denote by $\mathcal B(T^f_{\sigma})$ the set of leafs of $T^f_{\sigma}$. Similar $\mathcal B(T^f)$ is the set of leafs of $T^f$. 
    In addition, for all $\alpha<\kappa$ and for all $\sigma<\kappa$, let us define $(T^f_\sigma)^\alpha=T^f_\sigma\cap(A^f_\sigma)^\alpha$ and $(T^f)^\alpha=\bigcup_{\sigma<\alpha}(T^f_\sigma)^\alpha$.
\end{notation}

\begin{defn} Let $\varepsilon<\kappa$ be a regular cardinal.
\begin{itemize} 
\item $A\in K^{\lambda\cdot\lambda}_{tr}$ of size at most $\kappa$, is \textit{locally $(\kappa,\varepsilon)$-nice} if for every $\eta\in A\backslash P^A_{\lambda\cdot\lambda}$, $(Suc_A(\eta),<)$ is $(\kappa,\varepsilon)$-nice, $Suc_A(\eta)$ is infinite, and there is $\xi\in P_{\lambda\cdot\lambda}^A$ such that $\eta\prec\xi$.
\item $A\in K^{\lambda\cdot\lambda}_{tr}$ is \textit{$(<\kappa)$-stable} if for every $B\subseteq A$ of size smaller than $\kappa$, $$\kappa>|\{tp_{bs}(a,B,A)\mid a\in A\}|.$$
\end{itemize}
\end{defn}

\begin{lemma}\label{representation_tree}
    For any $f\in \kappa^\kappa$, $T^f$ is a locally $(\kappa, \varepsilon)$-nice and $(<\kappa)$-stable ordered tree and satisfies: 
    If $f,g$ are such that $f=_{S_0}^\kappa g$, then $T^f\cong T^g$.
\end{lemma}
\begin{proof}
    It is clear that the only thing left to prove is that for all $\eta\in T^f\backslash P^{T^f}_{\lambda\cdot\lambda}$ there is $\xi\in P_{\lambda\cdot\lambda}^{T^f}$ such that $\eta\prec\xi$. From Definition \ref{treeconst} (8), is enough to show that for all $\alpha<\kappa$ there is an increasing sequence $\langle \beta_i\mid i<\lambda\rangle$ such that from all $i<\lambda$, $\beta_i\in E_\sigma$ and $\bigcup_{i<\lambda}\beta_i\in S_\sigma^f$. This follows from the fact that $S_\sigma^f$ is a stationary set for all $\sigma<\kappa$ and $E_\sigma$ is a stationary set either equal to $S_\sigma^f$ or disjoint to $S_\sigma^f$. 
\end{proof}

The following rather simple lemma will prove itself useful for us. 

\begin{lemma}\label{the_tau}
    Let $\sigma<\beta<\kappa$ and $\eta$ be such that $\eta\in\mathcal B(T^f_{\sigma})\setminus (T^f)^{\beta}$. If there is no $\tau'<\lambda\cdot\lambda+1$ such that $\eta\restriction\tau'\in\acc((T^f)^{\beta})\setminus (T^f)^{\beta}$, then there exist $\tau<\lambda\cdot\lambda$ with the following properties:
    \begin{enumerate}
        \item $\eta\restriction\tau\in (T^f)^{\beta}$;
        \item $\eta\restriction(\tau+1)\notin (T^f)^{\beta}$;
        \item $\sup(\rng\eta_8\restriction\tau)<\beta$. 
    \end{enumerate}
\end{lemma}
\begin{proof}
      Let $\beta$, $\sigma$, and $\eta$ be as above. Let us denote $\delta:=\sup(\rng(\eta_8))$. Since $\eta\in\mathcal B(T^f_{\sigma})\setminus (T^f)^{\beta}$, $\delta>\beta$.
      Let $\tau:=\sup\{\tau'<\lambda\cdot\lambda+1\mid \sup(\rng\eta_8\restriction\tau')<\beta\}$.

Notice that $\tau<\lambda\cdot\lambda$, otherwise $\delta\leq\beta$. On the other side $\sup(\rng\eta_8\restriction\tau)\leq\beta$. Clearly $\rng_8\eta\restriction\tau=\beta$ if and only if $\eta\restriction\tau\in\acc((T^f)^{\beta})\setminus (T^f)^{\beta}$. By our assumptions, there is no $\tau'<\lambda\cdot\lambda+1$ such that $\eta\restriction\tau'\in\acc((T^f)^{\beta})\setminus (T^f)^{\beta}$ , so (1) and (2) follows.

Clearly (2) follows from the way we chose $\tau$.
\end{proof}

\section{Ehrenfeucht-Mostowski models}
To construct model of non-classifiable theories we will adapt the ideas from \cite{Sh}, \cite{HT}, and \cite{Mor23} to our particular ordered colored trees.

\begin{defn}
Let $\Delta$ be a set of formulas. Let $A$ and $\mathcal{M}$ be models, and $X=\{\bar{a}_s\mid s\in A\}$ an indexed set of finite tuples of elements of $\mathcal{M}$. We say that \textit{$X$ is a set of indiscernibles in $\mathcal{M}$ relative to $\Delta$}, if the following holds:

If $\bar{s},\bar{s}'$ are $n$-tuples of elements of $A$ and $tp_{at}(\bar{s},\emptyset,A)=tp_{at}(\bar{s}',\emptyset,A)$, then $$tp_\Delta(\bar{a}_{\bar{s}},\emptyset,\mathcal{M})=tp_\Delta(\bar{a}_{\bar{s}'},\emptyset,\mathcal{M}).$$
Here and from now on, $\bar{s}=(s_0,\ldots ,s_n)$ is a tuple of elements of $A$, and $\bar{a}_{\bar{s}}$ denotes $\bar{a}_{s_0}{}^\frown \cdots ^\frown\bar{a}_{s_n}$.
\end{defn}

\begin{defn}[Ehrenfeucht-Mostowski models]\label{EM_def}
    Let $\mathcal T$ be a $L_{\omega\omega}$-theory of vocabulary $\tau$, $l$ a dense linear order, $\mathcal{M}$ a model of vocabulary $\tau^1$, and $\varphi(\bar{u}, \bar{v})$ a formula in some logic $\mathcal{L}$.

    We say that \textit{$\mathcal{M}$ is an Ehrenfeucht-Mostowski model of $\mathcal T$ for $l$}, where the order is definable by $\varphi$, if $\mathcal{M}\models \mathcal T$, $\tau\subseteq \tau^1$, and there is a natural number $n$ and $n$-tuples of elements $\bar{a}_x\in \mathcal{M}$, $x\in l$, such that the following hold:
    \begin{enumerate}
        \item Every element of $\mathcal{M}$ is of the form $\mu(\bar{a}_{x_1},\ldots , \bar{a}_{x_m})$, where $\mu$ is a $\tau^1$-term and $x_1<\cdots <x_m$.
        \item If $x,y\in l$, then $\mathcal{M}\models \varphi(\bar{a}_x,\bar{a}_y)$ if and only if $x<y$.
        \item If $\psi(\bar{u}_1\ldots , \bar{u}_m)$ is an atomic $\tau^1$-formula, $x_1<\cdots <x_m$ and $y_1<\cdots <y_m$, then $$\mathcal{M}\models \psi(\bar{a}_{x_1},\ldots , \bar{a}_{x_m})\text{ iff }\mathcal{M}\models \psi(\bar{a}_{y_1},\ldots , \bar{a}_{y_m}).$$
    \end{enumerate}
\end{defn}

Suppose $\mathcal T$ is a theory such that for each dense linear order $l$, $\mathcal T$ has an Ehrenfeucht-Mostowski model where the order is definable by an $L_{\infty\omega}$-formula. We will only consider linear orders of some fixed set $B$. Let $l_B$ be a dense linear order such that every linear order of $B$ is a submodel of $l_B$. Let $EM_1(l_B)$, $\tau^1$, $\varphi$, $n$, $(\bar{a}_x)_{x\in l_B}$ be such that the conditions of Definition \ref{EM_def} are satisfied for $l_B$.

If $l\subseteq l_B$ is dense, then we define $EM_1(l)$ as the submodel of $EM_1(l_B)$ generated by $\bar{a}_x$, $x\in l$. Notice that $EM_1(l)$, $\tau^1$, $\varphi$, $n$, $(\bar{a}_x)_{x\in l}$ satisfy the conditions of Definition \ref{EM_def} for $l$.

We call the linear order \textit{$l$ the index model} of $EM_1(l)$. The indexed set \textit{$(\bar{a}_x)_{x\in l}$ is the skeleton of $EM_1(l)$}, and the tuples $\bar{a}_x$, $x\in l$, are the generating elements of $EM_1(l)$. Let us denote $EM(l)=EM_1(l)\restriction \tau$.

Suppose $\mathcal T$ is a theory such that for each dense linear order $l$, $\mathcal T$ has an Ehrenfeucht-Mostowski model where the order is definable by an $L_{\infty\omega_1}$-formula, and $B$ contains only $\omega_1$-dense linear orders. Then we can define $EM_1(l)$ and $EM(l)$ for all $l\in B$ as above.

\begin{fact}[{\cite[Fact 4.9]{Mor23}}]\label{DOP2}
For all $f\in \kappa^\kappa$, $T^f$ is $\varepsilon$-homogeneous with respect to quantifier free formulas.
\end{fact}

\begin{defn}\label{K_of_f}
    Let $\varepsilon\leq\lambda$ be a regular cardinal. If $\mathcal T$ is an unstable theory or a superstable theory with DOP. Then for every $f\in \kappa^\kappa$ define the order $K(f)$ by:
    \begin{enumerate}
        \item [I.] $dom \ K(f)=(dom\ T^f\times \{0\})\cup (dom\ T^f\times \{1\})$.
        \item [II.] For all $\eta\in T^f$, $(\eta,0)<_{K(f)}(\eta,1)$.
        \item [III.] If $\eta,\xi\in T^f$, then $\eta< \xi$ if and only if $(\eta,1)<_{K(f)}(\xi,0)$.
        \item [IV.] If $\eta,\xi\in T^f$, then $\eta\prec \xi$ if and only if $$(\eta,0)<_{K(f)}(\xi,0)<_{K(f)}(\xi,1)<_{K(f)}(\eta,1).$$ 

    \end{enumerate}
\end{defn}

In the case of $\mathcal T$ being a superstable theory with OTOP, we define $K(f)$ in a similar way modifying some items of Definition \ref{K_of_f}:
    \begin{enumerate}
        \item [I'.] $dom \ K(f)=(dom\ T^f\times \{0\})\cup \{(\eta,1)\mid\eta\in T^f\ \&\ T^f\not\models P_{\lambda\cdot\lambda}(\eta)\}$.
        \item [II'.] For all $\eta\in T^f$ such that $T^f\not\models P_{\lambda\cdot\lambda}(\eta)$, $(\eta,0)<_{K(f)}(\eta,1)$.
        \item [III'.] If $\eta,\xi\in T^f$, then $\eta< \xi$ if and only if $(\eta,1)<_{K(f)}(\xi,0)$.
        \item [IV'.] If $\eta,\xi\in T^f$ such that $T^f\not\models P_{\lambda\cdot\lambda}(\xi)\ \vee\ P_{\lambda\cdot\lambda}(\eta)$, then $\eta\prec \xi$ if and only if $$(\eta,0)<_{K(f)}(\xi,0)<_{K(f)}(\xi,1)<_{K(f)}(\eta,1).$$ 
        \item [V'.] If $\eta,\xi\in T^f$ such that $T^f\models P_{\lambda\cdot\lambda}(\xi)$ and $T^f\not\models P_{\lambda\cdot\lambda}(\eta)$, then $\eta\prec \xi$ if and only if $$(\eta,0)<_{K(f)}(\xi,0)<_{K(f)}(\eta,1).$$
    \end{enumerate}

\begin{fact}[{\cite[Lemma 4.14]{Mor23}}]\label{construction_unst}
Suppose $\mathcal T$ is a complete unstable theory in a countable relational vocabulary $\tau$. Let $\tau^1$ be a Skolemization of $\tau$, and $\mathcal T^1$ be a complete theory in $\tau^1$ extending $\mathcal T$ and with Skolem-functions in $\tau$. 
Then for every $f\in 2^\kappa$ there is $\mathcal{M}_1^f\models \mathcal T^1$ with the following properties.

\begin{enumerate}
    \item There is a map $\mathcal{H}: T^f\rightarrow (dom \ \mathcal{M}_1^f)^n$ for some $n<\omega$, $\eta\mapsto a_\eta$, such that $\mathcal{M}_1^f$ is the Skolem hull of $\{a_\eta\mid \eta\in T^f\}$. Let us denote $\{a_\eta\mid \eta\in T^f\}$ by $Sk(\mathcal{M}_1^f)$.
    \item $\mathcal{M}^f=\mathcal{M}_1^f\restriction \tau$ is a model of $\mathcal T$.
    \item $Sk(\mathcal{M}_1^f)$ is indiscernible in $\mathcal{M}_1^f$ relative to $L_{\omega \omega}$.
    \item There is a formula $\varphi\in L_{\omega \omega}(\tau)$ such that for all $\eta,\nu\in T^f$ and $m<{\lambda\cdot\lambda}$, if $T^f\models P_m(\eta) \wedge P_{\lambda\cdot\lambda}(\nu)$, then $\mathcal{M}^f\models \varphi (a_\nu,a_\eta)$ if and only if $T^f\models \eta\prec \nu$.
\end{enumerate}
\end{fact}

There are corespondent Facts for superstable theories with OTOP or DOP. In the OTOP case the only change is  $L_{\omega \omega}$ for $L_{\infty \omega}$. In the DOP case, $L_{\omega \omega}$ changes for $L_{\omega_1 \omega_1}$ and $\varepsilon=\omega_1$. In particular, in the case of DOP, $\omega_1\leq\lambda$.

Let us define $\mathbb F^*=\{\eta\in \mathbb F\mid \forall\alpha>0 (\eta(\alpha)<\alpha)\}$.
Let $\mathcal{T}$ be an unstable theory or a superstable non-classifiable theory. For all $f\in \mathbb F^*$, let $\mathcal{M}^f$ be the model from Fact \ref{construction_unst} (or the respective one in the case of OTOP or DOP). 

The construction of Ehrenfeucht-Mostowski models for stable unsuperstable theories is different than for the other cases. The construction of the models follow the generalized Ehrenfeucht-Mostowski models (see \cite{Sh} Theorem 1.3 and \cite{Sh90} Chapter VII). The construction of such models depends on the existence of a proper function $\Phi$. Unfortunately, the argument in \cite{Sh90} for the existence of a proper function, fails when the ordered tree has high $\omega\cdot\omega$ or higher. 

\begin{defn}
	Let $T\in K^{\lambda\cdot\lambda}_{tr}$, $A\subseteq T$ and $\eta\in T$.
	\begin{itemize}
	\item We say that $A$ is \textit{downward closed} if for all $\eta\in A$, $\eta\restriction m\in A$ if $m<lg(\eta)$.
	\item Let $\eta^{\downarrow}=(\eta\restriction \alpha)_{\alpha\leq lg(\eta)}$.
	\end{itemize}
\end{defn}

\begin{lemma}\label{similar type lemmaf}
Let $T\in K^{\lambda\cdot\lambda}_{tr}$, $A\s T$ downward closed. Let $\vec{\eta}_0:=\langle\eta^i_0\mid i<n\rangle $ and $\vec{\eta}_1:=\langle\eta^i_1\mid i<n\rangle$ two sequences of elements in $T$. If the following holds, 
\begin{enumerate}
    \item $\tp_{\bs}(\vec{\eta}_0,\emptyset,T)=\tp_{\bs}(\vec{\eta}_1,\emptyset,T)$; 
    \item for all $i<n$, $\tp_{\bs}((\eta^i_0)^{\downarrow},A,(T,\prec,<))=\tp_{\bs}((\eta^i_1)^{\downarrow},A,(T,\prec,<))$.
\end{enumerate}
Then,
$$\tp_{\bs}(\vec{\eta}_0,A,T)=\tp_{\bs}(\vec{\eta}_1,A,T).$$
\end{lemma}
\begin{proof}
We define the map $\phi:A\cup(\vec{\eta}_0)^{\downarrow}\rightarrow A\cup(\vec{\eta}_1)^{\downarrow}$ as follows:
If $\eta\in A$, then $\phi(\eta)=\eta$. Otherwise, for $i<n$, $\tau<\dom(\eta^i_0)$ define $\phi(\eta^i_0\restriction\tau)=\eta^i_1\restriction\tau$.
\begin{claim}
$\phi$ is a well defined bijective map.
\end{claim}
\begin{proof}
Suppose, $\eta\in A$ and for some $i<n$ and $\tau<\dom(\eta^i_0)$, $\eta=\eta^i_0\restriction\tau$. Then, by Clause~(2), $T\models \eta=\eta^i_1\restriction\tau$.
Suppose now, that for distinct $i,j<n$ and $\tau<\dom(\eta^i_0)$ $\eta_0^i\restriction\tau=\eta^j_0\restriction\tau$.
Then, there exist $\tau'\geq\tau$ such that $T\models P_{\tau'}(\wedge(\eta^i_0,\eta^j_0))$.
By Clause~(1), $T\models P_{\tau'}(\wedge(\eta^i_0,\eta^j_0))$ and $\eta^i_1\restriction\tau=\eta^j_1\restriction\tau$.
By a symmetrical argument the map $\phi^{-1}$ is well defined \qedhere
\end{proof}
Since both $\dom(\phi),\rng(\phi)$ are closed under $\wedge$, they are submodels of $T$. Thus, if we manage to prove the following claim we are done.
\begin{claim}
The map $\phi$ preserves the relations $\prec$ and $<$. 
\end{claim}
\begin{proof}
Let $\eta,\nu\in\dom(\phi)$. Note that for any relation $R\in\{\prec,<\}$ if both $\eta,\nu\in A$ then the conclusion is trivial. So, assume that $\eta\notin A$. Let $i<n$ and $\tau<\dom(\eta^i_0)$ be such that $\eta=\eta^i_0\restriction\tau$. 
We split into two cases depending on $R$:

$\br$ Suppose that $\eta\prec\nu$. Observe that $\nu\notin A$, otherwise by the downward closure of $A$ we get $\eta\in A$ as well. 
Therefore, there exist $j<n$ and $\tau'\leq\dom(\eta^j_0)$ such that $\nu=\eta^j_0\restriction\tau'$. Since $\eta\prec\nu$, $\eta^j_0\restriction\tau=\eta$. Therefore, $\eta^j_0\restriction\tau=\eta^i_0\restriction\tau$.  

Altogether,
$$\phi(\eta)=\phi(\eta^i_0\restriction\tau)=\phi(\eta^j_0\restriction\tau)=\eta^j_1\restriction\tau\prec\eta^j_1\restriction\tau'=\phi(\eta^j_0\restriction\tau')=\phi(\nu).$$

$\br$ Suppose now that $\eta<\nu$. If $\nu\in A$, then the conclusion follows from Clause~(2) of the lemma. 
Hence, suppose that $\nu\notin A$. Let $j<n$ and $\tau'<\dom(\eta^j_0)$ such that, $\tau=\tau'+1$, $\eta=\eta^i_0\restriction\tau$, $\nu=\eta^j_0\restriction\tau$, and $\eta\restriction\tau'=\nu\restriction\tau'$. 
Since $T\models P_{\tau'}(\wedge(\eta,\nu))$ and $T\models\eta<\nu$, 
by clause (1) $T\models P_{\tau'}(\wedge(\eta^i_1\restriction\tau,\eta^j_1\restriction\tau))$ and $T\models\eta^i_1<\eta^j_1$. 
Therefore, 
$$\phi(\eta)=\eta^i_1\restriction\tau<\eta^j_1\restriction\tau=\phi(\nu).$$
\end{proof}
The proof is now complete.
\end{proof}

\begin{lemma}\label{iso_direction}
Let $f,g\in \mathbb F^*$. If $\mathcal{M}^f$ and $\mathcal{M}^g$ are isomorphic, then 
$f=^\mathbb F_{S_0} g$.

\end{lemma}
\begin{proof}
Let us assume, for the sake of contradiction, that $\mathcal M^f$ and $\mathcal M^g$ are isomorphic but $f\neq^\mathbb F_{S_0} g$.
Thus, the set $S:=\{\delta\in S_0\mid f(\delta)\neq g(\delta)\}$ is stationary. 
Since $f\in \mathbb F^*$, applying Fodor's lemma on $S$, there exist $S'\s S$ stationary and $\sigma<\kappa$ such that $S'\s f^{-1}\{\sigma\}$ and $S^f_\sigma\neq E_\sigma$.

\begin{claim}
    $S'\cap S^g_{\sigma}=\emptyset$.
\end{claim}
\begin{proof}
    By definition, $S^g_{\sigma}\in\{g^{-1}\{\sigma\},E_{\sigma}\}$. Since $S'\s f^{-1}\{\sigma\}\cap S$, $S'\cap g^{-1}\{\sigma\}=\emptyset$. On the other hand, $S'\cap E_{\sigma}\subseteq S_0\cap S_1=\emptyset$.
\end{proof}

Therefore, $S' \triangle S^g_\sigma$ is stationary, so $S^f_\sigma \triangle S^g_\sigma$ is stationary. For all $\sigma'<\kappa$, we will write $S^f_{\sigma'}\neq^{\ns}S^g_{\sigma'}$ when $S^f_{\sigma'} \triangle S^g_{\sigma'}$ is stationary. Thus, $\{\sigma'<\kappa\mid S^f_{\sigma'}\neq^{\ns}S^g_{\sigma'}\}$ is nonempty, let $\Sigma:=\min\{\sigma'<\kappa\mid S^f_{\sigma'}\neq^{\ns}S^g_{\sigma'}\}$.

Notice that, either $S^g_{\Sigma}\backslash S^f_{\Sigma}\s S_0$ is stationary, or $S^f_{\Sigma}\backslash S^g_{\Sigma}\s S_0$ is stationary. Since otherwise, $S^{g}_{\Sigma}=E_{\Sigma}=S^f_{\Sigma}$ contradicting $S^{g}_{\Sigma}\neq^{\ns}S^f_{\Sigma}$. 
Without loss of generality, we may assume that $S^f_{\Sigma}\s S_0$ and $S^f_{\Sigma}\backslash S^g_{\Sigma}$ is stationary. Let us denote by $\sigma_f:=\Sigma$. 
From the minimality of $\sigma_f$, for all $\sigma<\sigma_f$, $S^{g}_{\sigma}=^{\ns} S^{f}_{\sigma}$ and $S^{f}_{\sigma}\cap S^{f}_{\sigma_f}=\emptyset$.
Therefore, there is a stationary subset $S^f\s S^f_{\sigma_f}$ such that $S^f\cap(\bigcup_{\sigma\leq \sigma_f} S^{g}_\sigma)=\emptyset$. 

Since $g$ is regressive and $S^f$ is stationary, by Fodor's lemma there is $\sigma<\kappa$ such that $S^f\cap g^{-1}\{\sigma\}$ is stationary.
Let $\sigma_g:=\min\{\sigma<\kappa\mid S^f\cap g^{-1}\{\sigma\}\text{ is stationary }\}$.
Since $S^f\cap(\bigcup_{\sigma\leq\sigma_f} S^{g}_\sigma)$ is empty, we know $\sigma_g>\sigma_f.$
Set $S^g:=S^f\cap g^{-1}\{\sigma_g\}$. 

Let $F$ be an isomorphism from $\mathcal M^g$ to $\mathcal M^f$.
Let us denote by $\bar a_\eta$ and $\bar b_\xi$ the elements of $Sk(\mathcal M_1^g)$ and $Sk(\mathcal M_1^f)$. For a sequence $\bar{a}=(a^0,\ldots, a^m)$ from $\mathcal M^g$ we denote $F(\bar{a})=(F(a^0),\ldots,F(a^m))$ and for a sequence $\bar{v}=(v^0,\ldots,v^m)$ from $T^f$ we denote $\bar{b}_{\bar{v}}=\bar b_{v^0}^\frown\cdots^\frown \bar b_{v_m}$. For each $\eta\in T^g$ let $$F(\bar{a}_\eta)=(\mu_\eta^0(\bar b_{\bar{v}_\eta}),\ldots,\mu_\eta^m(\bar b_{\bar{v}_\eta}))=\bar{\mu}_\eta(\bar{b}_{\bar{v}_\eta}),$$ where $m=lg(\bar a_\eta)-1$, $\mu_\eta^i$ are $\tau^1$-terms and $\bar{v}_\eta$ is a finite sequence of elements of $T^f$.
Let $\Pi:T^g\rightarrow[T^f]^{<\omega}$ be the map given by $$F(\bar{a}_\eta)=\bar{\mu}_\eta(\bar{b}_{\Pi(\eta)}).$$

Recall Notation \ref{notation1} for the definition of $\mathcal B(T^g)$ and $(T^f)^\alpha$.

The following claim is the key of our proof:
\begin{claim}\label{claim of beta}
    For every $\eta\in\mathcal B(T^g)$ with $\delta:=\sup(\rng(\eta_8))$ in $S^g$, there exist $\xi\preceq\eta$ in $T^g$ such that, for all $i<\len(\Pi(\eta))$ for all $\tau\leq\lambda\cdot\lambda$, for $\beta:=\sup(\rng(\xi_8))$ $\Pi(\eta)(i)\restriction\tau\notin\acc((T^f)^{\beta})\setminus (T^f)^{\beta}$.
    In addition,
     if $\Pi(\eta)(i)\in (T^f)^{\delta}$, then $\Pi(\eta)(i)\in (T^f)^{\beta}$.

\end{claim}
\begin{proof}
Let $\eta\in\mathcal B(T^g)$ with $\delta:=\sup(\rng(\eta_8))$ in $S^g$. Denote by $n:=\len(\Pi(\eta))$ and for each $i<n$, let $\nu_i:=\Pi(\eta)(i)$. 
For each $i<n$, we define $B_i\s\delta$ by the following cases:
\begin{itemize}
    \item[(1)] $\nu_i\in (T^f)^{\delta}$: From the definition of the filtration, there is $\beta<\delta$ such that $\nu_i\in (T^f)^{\beta}$. Let $B_i$ be the singleton of the minimal ordinal $\beta$ such that $\nu_i\in (T^f)^{\beta}$. 
    \item[(2)] $\nu_i\in (T^f)^{\bar{\delta}}$ for some $\bar{\delta}>\delta$ and there is $\bar \tau\leq \dom(\nu_i)$ such that $\nu_i\restriction \bar \tau\in\acc((T^f)^{\delta})\setminus (T^f)^{\delta}$. 
    There is $k\leq 8$ $\sup(\rng((\nu_i\restriction \bar \tau)_k))=\delta$. 
    By Remark~\ref{connection to sup of eta5}, $\sup(\rng((\nu_i\restriction \bar \tau)_8))=\delta$.
    Let $B_i:=\acc(\rng((\nu_i)_8\restriction \bar \tau))$.

    \item[(3)] Else, $\nu_i\in (T^f)^{\bar{\delta}}$ for some $\bar{\delta}>\delta$ 
    and there is no $\tau\leq\dom(\nu_i)$ 
    such that $\nu_i\restriction \tau\in\acc((T^f)^{\delta})\setminus(T^f)^{\delta}$. Let $B_i:=\acc(\rng((\nu_i)_8\restriction\bar{\tau}))$, where $\bar{\tau}$ is as in Lemma~\ref{the_tau}.
\end{itemize}
Notice that besides Clause (2), we have $\sup(B_i)<\delta$. 

For all $i<n$ satisfying Clause (2) and $\bar\tau\leq \dom(\nu_i)$ witnessing it, $\sup(\rng((\nu_i)_8\restriction\bar \tau))=\delta$. Since $\delta\in S^g$, $\delta\notin S_1$ and $\bar \tau= \dom(\nu_i)=\lambda\cdot\lambda$. Thus $\nu_i\in\mathcal B(T^f_{\sigma_f})$. 
Recall that $\sigma_f\neq\sigma_g$. Thus, by the construction of $T^f$ specifically Definition \ref{treeconst} clause~(8),
$$B_i\cap\acc(\rng(\eta_8))\s (E_{\sigma_f}\cup\{\delta\})\cap(E_{\sigma_g}\cup\{\delta\})=\{\delta\}.$$ 

Let us choose $\beta\in\acc(\rng(\eta_8))\cap S_\lambda$ such that for all $i<n$ satisfying (1) or (3), $\sup(B_i)<\beta<\delta$, and set $\xi\preceq\eta$ such that $\beta=\sup(\rng(\xi_8))$

Let us show that $\xi$ is the desired node.
Fix $i<n$ and $\tau\leq\lambda\cdot\lambda$ and denote $\zeta:=\nu_i\restriction\tau$.

$\br$ Case Clause~(1)  holds for $i$: By the 
definition of $B_i$, $\zeta\in(T^f)^{\beta}$.

$\br$ Case Clause~(3) holds for $i$: By the definition of $B_i$ and by the
choice of $\bar{\tau}$, there are two subcases:
\begin{itemize}
    \item $\tau\leq\bar{\tau}$. We have $\zeta\subseteq \nu_i\restriction\bar\tau\in (T^f)^{\beta}$,
    \item $\tau\ge\bar{\tau}$. We have $\nu_i\restriction\bar\tau\subsetneq\zeta$, $\nu_i\restriction\bar\tau\in (T^f)^{\delta}$ and $\nu_i\restriction\bar\tau+1\notin (T^f)^{\delta}$. So $\zeta\notin\acc((T^f)^{\beta})\setminus(T^f)^{\beta}$.
\end{itemize}

$\br$ Case Clause~(2) holds for $i$: Let us suppose, for the sake of contradiction, that $\zeta\in\acc((T^f)^{\beta})\setminus (T^f)^{\beta}$.
Therefore, by Remark~\ref{connection to sup of eta5}
$\sup(\rng(\zeta)_8)=\beta$ and $\beta\in\acc(\rng(\nu_i)_8)$. 
Recall that $\delta\in S^g$. So $\delta\in S^f\cap g^{-1}\{\sigma_g\}$, and since $S^g$ is stationary, $g^{-1}\{\sigma_g\}$ is stationary. We conclude that $\delta\in S^g_{\sigma_g}$ and therefore by construction, $\acc(\rng(\eta)_8)\subseteq E_{\sigma_g}\cup \{\delta\}$. Thus $\beta\in E_{\sigma_g}$.

On the other hand, $\delta\in S^f$ so $\delta\in S^f_{\sigma_f}$.
As $\nu_i\restriction\bar\tau\in\acc((T^f)^{\delta})\setminus (T^f)^{\delta}$ and the analysis above: 
\begin{enumerate}
\item[(4)] $\bar\tau=\lambda\cdot\lambda$ and $\sup(\rng((\nu_i\restriction\bar\tau)_8))=\delta$;
    \item[(5)] $\nu_i\in \mathcal{B}(T^f_{\sigma_f})$;
    \item[(6)] $\acc(\rng((\nu_i)_8)\subseteq E_{\sigma_f}\cup \{\delta\}$.
    
\end{enumerate}
Since $\beta\in\acc(\rng((\nu_i)_8))\setminus\{\delta\}$, $\beta\in E_{\sigma_f}$.

Altogether,  $\beta\in E_{\sigma_f}\cap E_{\sigma_g}$. But as $\sigma_f\neq\sigma_g$, $E_{\sigma_f}\cap E_{\sigma_g}=\emptyset$ which is a contradiction. \qedhere

\end{proof}

For all $\eta\in T^f$, $T^f\not\models P_{\lambda\cdot\lambda} (\eta)$, and $\alpha\leq\kappa$ let $$B^f(\eta,\alpha)=Suc_{T^f}(\eta)\cap (T^f)^\alpha.$$ 
It is clear that $\langle B^f(\eta,\alpha) \mid \alpha<\kappa\rangle$ is a filtration of $Suc_{T^f}(\eta)$. 
By Theorem \ref{representation_tree}, $T^f$ is $(\kappa, \varepsilon)$-nice, in particular $Suc_{T^f}(\eta)$ is isomorphic to $\mathcal I$. Since any two representations coincide in a club, for any $\eta\in T^f$ there is a club $C_\eta$ such that for all $\delta\in C_\eta$ with $\cf(\delta)\ge\varepsilon$ and $\nu\in Suc_{T^f}(\eta)$ there is $\beta<\delta$ such that 
$$\forall u\in B^f(\eta,\delta)\ [u>\nu \Rightarrow \exists u'\in B^f(\eta,\beta)\ (u\ge u'\ge\nu)].$$
Let $$\bar{C}^f=\{\delta<\kappa\mid\cf(\delta)\ge\varepsilon\text{ and for all }\eta\in (T^f)_\delta,\  \delta\in C_\eta\}$$ and $C^f$ be $\bar{C}^f$ closed under $\alpha$-limits for $\alpha<\varepsilon$.
Notice that $C^f$ is a club that satisfies that for all $\delta\in C^f$ with $\cf(\delta)\ge\varepsilon$, $\eta\in T^f$, $T^f\not\models P_{\lambda\cdot\lambda} (\eta)$, and $\nu\in Suc_{T^f}(\eta)$ there is $\beta<\delta$ such that,
$$\forall u\in B^f(\eta,\delta)\ [u>\nu \Rightarrow \exists u'\in B^f(\eta,\beta)\ (u\ge u'\ge\nu)].$$

Let us define $C^g$ in a similar way.

Let 
\begin{itemize}
    \item $C_0=C^f\cap C^g$.
    \item $C_1:=\{\delta\in C_0\mid\forall\eta\in T^g(\eta\in (T^g)^{\delta}\text{ implies }\Pi(\eta)\subseteq (T^f)^{\delta})\}.$
    \item $C'_2:=\{\delta\in C_1\mid \forall\alpha<\delta\ \forall\eta\in (T^g)^{\delta}\ \forall u_0\in B^g(\eta,\kappa)\ \exists u_1\in B^g(\eta,\delta)$ $$ \left[\Pi(u_0),\Pi(u_1)\text{ have the same atomic type over }(T^f)^{\alpha}\text{ and }\bar{\mu}_{u_0}=\bar{\mu}_{u_1}\right]\}.$$
\item $C_2=\{\delta\in C'_2\mid\cf(\delta)\ge\lambda\}$
\item $C=\{\delta\in C_2\mid \delta\in C_2\ \&\ \delta\text{ is a limit point of }C_2\}.$ 
\end{itemize}

It is clear that $C_0$ and $C_1$ are clubs. Since $T^f$  is $(<\kappa)$-stable ordered tree, there are less than $\kappa$ possible $bs$-types of $\Pi (u_0)$ over $(T^f)^\alpha$, and since $|\tau^1|<\kappa$, there are less than $\kappa$ possible terms $\bar{\mu}_{u_0}$, so $C'_2$ is a club.

Thus, 

(*) \textit{for all $\delta\in C$, $\eta\in T^f$, $T^f\not\models P_{\lambda\cdot\lambda} (\eta)$, and $\nu\in Suc_{T^f}(\eta)$ there is $\beta<\delta$ such that,
$$\forall u\in B^f(\eta,\delta)\ [u>\nu \Rightarrow \exists u'\in B^f(\eta,\beta)\ (u\ge u'\ge\nu)].$$}

Since $S^g\subseteq S_0$, and $S_1$ and $S^g$ are stationary, $S_1\cap S^g=\emptyset$ and there are $\delta\in S^g\cap C$ and $\eta\in T^g$, such that:
\begin{itemize}
\item [I.] $T^g\models P_{\lambda\cdot\lambda}(\eta)$.
\item [II.] $\delta=\sup(\eta_8)$.
\item [III.] For all $n<\lambda\cdot\lambda$, $\eta\restriction n\in (T^g)^\delta$.
\item [IV.] For all $\alpha<\delta$, there is $m<\lambda\cdot\lambda$ such that $\eta\restriction m\notin (T^g)_\alpha$.
\item [V.] $\acc(\rng(\eta_8))\cap S_{\ge\lambda}\s C$.
\end{itemize}
 
By Claim \ref{claim of beta} we may find $\eta'\in T^g$ such that,

\begin{enumerate}
    \item $\eta'\preceq\eta$;
    \item $\beta:=\sup(\rng(\eta'_8))$;
    \item for all $i<\len(\Pi(\eta))$ for all $\tau\leq\lambda\cdot\lambda$, $\Pi(\eta)(i)\restriction\tau\notin\acc((T^f)^{\beta})\setminus (T^f)^{\beta}$.
\end{enumerate}
Notice from the proof of Claim \ref{claim of beta}, that $\cf(\beta)=\lambda$. From V. in the selection of $\delta$ and $\eta$, we have $\beta\in C$. In particular it is a $\lambda$-limit of $C_2$. 

Denote by $n:=\len(\Pi(\eta))$ and let $\langle\nu_i\mid i<n\rangle$ be an enumeration of $\Pi(\eta)$. 
For all $i<n$ we choose ordinals $\alpha_i\in C_2\cap\delta$ and $\tau_i\leq\lambda\cdot\lambda$ according the following cases:
\begin{itemize}
    \item If $\nu_i\in (T^f)^{\delta}$, then by Claim~\ref{claim of beta}, $\nu_i\in (T^f)^{\beta}$.
    Now, the definition of $\beta$, implies the existence of $\alpha\in C_2\cap\beta$ such that, $\nu_i\in (T^f)^{\alpha}$. Set $\alpha_i:=\alpha$, and $\tau_i:=\dom(\nu_i)$.
    \item If $\nu_i\notin (T^f)^{\delta}$, then by Clause (3) above $\nu_i\notin\acc((T^f)^{\beta})$. Therefore $\nu_i\in (T^f)^{\delta'}$ for some $\delta'>\delta>\beta$.
    So, since $\beta\in C$, by (*), and by addressing Lemma~\ref{the_tau} with $\nu_i$ and $\beta$, we may find some ordinal $\alpha\in C_2\cap\beta$ and $\tau\leq\len(\nu_i)$ such that:
    \begin{itemize}
        \item $w^0=\nu_i\restriction\tau\in (T^f)^{\alpha}$ but $w^1=\nu_i\restriction(\tau+1)\notin (T^f)^{\beta}$.
        \item (**) $\forall u\in B^f(w^0,\beta)\ [u>w^1 \Rightarrow \exists u'\in B^f(w^0,\alpha)\ (u\ge u'\ge w^1)]$.
    \end{itemize}
    Set $\alpha_i:=\alpha$, $\tau_i:=\tau$. 
\end{itemize}

Denote $\alpha:=\max\{\alpha_i\mid i<n\}$. Since $\beta\in C$, we may pick $\gamma\in C_2\cap\beta$ above $\alpha$. 
Let $\tau<\lambda\cdot\lambda$ be the maximal ordinal such that $\eta\restriction\tau\in (T^g)^{\gamma}$. 
Set $\eta_0:=\eta\restriction(\tau+1)$, so $\eta_0\notin (T^g)^{\gamma}$. 
Since $\gamma\in C_2$, we can find $\eta_1$ with the following properties:
\begin{itemize}
    \item [$\mathbf{a}$.] $\eta_1\in(T^g)^{\gamma}$
    \item [$\mathbf{b}$.] $\dom(\eta_1)=\tau+1$;
    \item [$\mathbf{c}$.] $\eta_0\restriction \tau=\eta_1\restriction\tau$;
    \item [$\mathbf{d}$.] $\bar{\mu}_{\eta_0}=\bar{\mu}_{\eta_1}$;
    \item [$\mathbf{e}$.] $\Pi(\eta_0),\Pi(\eta_1)$ have the same atomic type over $(T^f)^{\alpha}$. 
\end{itemize}
Notice that from $\mathbf{a-c}$ we can deduce that $\eta_1\in B^g(\eta\restriction\tau, \gamma)$. Since $\beta\in C_1$, $\Pi(\eta_0),\Pi(\eta_1)\in (T^f)^{\beta}$.

\begin{claim}
$\tp_{at}(\Pi(\eta_0)^{\smallfrown}\Pi(\eta),\emptyset,T^f)=\tp_{at}(\Pi(\eta_1)^{\smallfrown}\Pi(\eta),\emptyset,T^f)$. 
\end{claim}
\begin{proof}
    By Lemma~\ref{similar type lemmaf} it is enough to prove that for all $i<n$,
     $$\tp_{\bs}((\Pi(\eta_0)_i)^{\downarrow},(\Pi(\eta)_i)^{\downarrow},(T^f,\prec,<))=\tp_{\bs}((\Pi(\eta_1)_i)^{\downarrow},(\Pi(\eta)_i)^{\downarrow},(T^f,\prec,<)).$$
     Let $i<n$ be fixed. Let $\nu=\Pi(\eta)_k\restriction r\in(\Pi(\eta)_k)^{\downarrow}$ and $\nu_0=\Pi(\eta_0)_i\restriction r_0\in (\Pi(\eta_0)_i)^{\downarrow}$.

        $\underline{\nu_0=\nu}$: Since $\nu_0\in (T^f)^{\beta}$ we have $\nu\in (T^f)^{\beta}$. 
         By the choice of $\alpha$, $\nu\in (T^f)^{\alpha}$ and hence $\nu_0\in (T^f)^{\alpha}$. By $\mathbf{e}$, $\nu_0\preceq\Pi(\eta_1)_i$. In particular, $\nu_0=\nu=\Pi(\eta_1)_i\restriction r_0$.

         \underline{$\nu_0,\nu$ are $\prec$ compatible}: 
         If $\nu_0\prec\nu$, then as in the previous case, we can conclude that $\nu_0\in (T^f)^{\alpha}$. By $\mathbf{e}$, we get that $\Pi(\eta_1)_i\restriction r_1=\nu_0\prec\nu$.

         If $\nu\prec\nu_0$, then since $\nu_0\in (T^f)^{\beta}$, $\nu\in (T^f)^{\beta}$. So $\nu\in (T^f)^{\alpha}$. By $\mathbf{e}$, we conclude that $\nu\prec\Pi(\eta_1)_i\restriction r_0$.

        \underline{$\nu_0,\nu$ are $<$ compatible}:
        Since $\nu_0,\nu$ are $<$ compatible, $r=r_0$. If $\nu\in (T^f)^{\alpha}$, then by $\mathbf{e}$, $\nu_0<\nu$ implies $\Pi(\eta_1)_i\restriction r_0<\nu$, and $\nu<\nu_0$ implies $\nu<\Pi(\eta_1)_i\restriction r_0$

         Let us suppose that $\nu\notin (T^f)^{\alpha}$, so $\nu\notin (T^f)^{\delta}$, $r=\tau+1$, and $\nu_0\restriction\tau=\nu\restriction\tau\in (T^f)^{\delta}$. 
         Let us suppose, towards contradiction, that $\nu_0<\nu< \Pi(\eta_1)_i\restriction r_0$ or $\Pi(\eta_1)_i\restriction r_0<\nu<\nu_0$ holds. 
         
         Let us show the latest case, the other case is similar. Since $$\nu_0\subseteq\Pi(\eta_0)_i\restriction r_0\subseteq \Pi(\eta_0)_i\in (T^f)^\beta.$$ Applying (**) with $u=\nu_0$, $w^0=\nu\restriction\tau$, and $w^1=\nu$, we have that there is $u'\in B^f(\nu\restriction\tau,\alpha_i)$ such that $\nu\leq_Iu'\leq_I \nu_0$. Thus $\Pi(\eta_1)_i\restriction r_0<\nu\leq_Iu'\leq_I\nu_0$, but $u'\in (T^f)^{\alpha}$, a contradiction with $\mathbf{e}$.
\end{proof}

From the previous claim and the way the models were constructed (recall the formula $\varphi$ from Fact \ref{construction_unst}), we know
$$\mathcal{M}^f_1\models\varphi(\bar{\mu}_\eta(\bar{b}_{\Pi(\eta)})^{\smallfrown}\bar{\mu}_{\eta_0}(\bar{b}_{\Pi(\eta_0)}))\Leftrightarrow\varphi(\bar{\mu}_\eta(\bar{b}_{\Pi(\eta)})^{\smallfrown}\bar{\mu}_{\eta_1}(\bar{b}_{\Pi(\eta_1)}))$$
so
$$\mathcal{M}^g_1\models\varphi(\bar{a}_{\eta}^{\smallfrown}\bar{a}_{\eta_0})\Leftrightarrow(\varphi(\bar{a}_{\eta}^{\smallfrown}\bar{a}_{\eta_1})).$$
But, since $\eta_0\preceq\eta$ and $\eta_1\not\preceq\eta$, then 
$$\mathcal{M}^g\models\varphi(\bar{a}_{\eta}^{\smallfrown}\bar{a}_{\eta_0})\wedge\neg(\varphi(\bar{a}_{\eta}^{\smallfrown}\bar{a}_{\eta_1})).$$
A contradiction, since $\mathcal{M}^g=\mathcal{M}^g_1\restriction\tau$.\qedhere

\end{proof}

\section{Conclusions}

Let us proceed to the proof of the main result. 
\begin{lemma}\label{red_cor}
Let $\mathcal T$ be a countable complete theory over a countable relational vocabulary, and unstable or superstable non-classifiable. Let $\lambda=\omega_1$ if $\mathcal{T}$ has the DOP, otherwise $\lambda=\omega$. There is a continuous function $\mathcal{F}:\mathbb F^*\rightarrow\kappa^\kappa$ such that for all $\eta,\xi\in \mathbb F$, $\eta=^\mathbb F_{S_0}\xi$ if and only if $\mathcal F(\eta)\cong_\mathcal T \mathcal F(\xi)$.    
\end{lemma}
\begin{proof}
    From Lemma \ref{representation_tree} and Lemma \ref{iso_direction}, for all $f,g\in \mathbb{F}^*$, $\mathcal{M}^f$ and $\mathcal{M}^g$ are isomorphic if and only if $f=^{\mathbb F}_{S_0} g$. 

    For every $f\in \mathbb F^*$, we will construct a model $\mathcal{M}_f$ isomorphic to $\mathcal{M}^f$. We will also construct a function $\mathcal{G}:\{\mathcal{M}_f\mid f\in 2^\kappa\}\rightarrow 2^\kappa$, such that $\mathcal{A}_{\mathcal{G}(\mathcal{M}_f)}\cong \mathcal{M}_f$ and $f\mapsto \mathcal{G}(\mathcal{M}_f)$ is continuous. 
    
    Recall the definition of $(T^g)^\alpha$ from Notation \ref{notation1}. Since $f\in \mathbb F^*$ and for all $\alpha,\sigma<\kappa$ $f\restriction\alpha=g\restriction\alpha$ if and only if $(T_\sigma^f)^\alpha=(T_\sigma^g)^\alpha$, $$f\restriction \alpha=g\restriction \alpha \Leftrightarrow (T^f)^\alpha=(T^g)^\alpha.$$ 
    
For all $\alpha<\kappa$, $A\in K^{\lambda\cdot\lambda}_{tr}$, and a filtration $\mathbb{A}=\langle A^\iota\mid \iota<\kappa  \rangle$ of $A$, let us denote by $\tilde{A}^\alpha$ the set $\{a_s\mid s\in A^\alpha\}$, recall the construction of the models $\mathcal{M}^f_1$. 
Since for all $\alpha<\kappa$, $$(T^f)^\alpha=(T^g)^\alpha \Leftrightarrow SH(\tilde{(T^f)}^\alpha) = SH(\tilde{(T^g)}^\alpha),$$
for all $f\in \mathbb F^*$ we can construct a tuple $(\mathcal{M}_f, F_f)$, where $\mathcal{M}_f$ is a model isomorphic to $\mathcal{M}^f$ and $F_f:\mathcal{M}_f\rightarrow \mathcal{M}^f$ is an isomorphism, that satisfies the following: denote by $\mathcal{M}_{f,\alpha}$ the preimage $F^{-1}_f[SH(\tilde{(T^f)}^\alpha)\restriction \tau]$ and $$f\restriction \alpha=g\restriction \alpha \Leftrightarrow \mathcal{M}_{f,\alpha}=\mathcal{M}_{g,\alpha}.$$
For every $f\in \mathbb F^*$ there is a bijection $E_f: \dom(\mathcal{M}_f)\rightarrow \kappa$, such that for every $f,g\in \mathbb F^*$ and $\alpha<\kappa$, if $f\restriction \alpha=g\restriction \alpha $, then $E_f\restriction \dom(\mathcal{M}_{f,\alpha})=E_g\restriction \dom(\mathcal{M}_{g,\alpha})$ (see \cite{Mor}). 
Let us denote by $\pi$ the bijection from Definition \ref{struct}, define the function $\mathcal{G}$ by: 
$$\mathcal{G}(\mathcal{M}_f)(\alpha)=\begin{cases} 1 &\mbox{if } \alpha=\pi(m,a_1,a_2,\ldots,a_n) \text{ and }\\ & \mathcal{M}_{f}\models Q_m(E_f^{-1}(a_1),E_f^{-1}(a_2),\ldots,E_f^{-1}(a_n))\\
0 & \mbox{otherwise. } \end{cases}$$
To show that $G:\mathbb F\rightarrow 2^\kappa$, $G(f)=\mathcal{G}(\mathcal{M}_f)$ is continuous, let $[\zeta\restriction \alpha]$ be a basic open set and $\xi\in G^{-1}[[\zeta\restriction \alpha]]$. There is $\beta<\kappa$ such that for all $\epsilon<\alpha$, if $\epsilon=\pi(m,a_1,\ldots, a_n)$, then $E^{-1}_\xi (a_i)\in \dom(\mathcal{M}_{\xi,\beta})$ holds for all $i\leq n$. Since for all $\eta\in [\xi\restriction\beta]$ it holds that $\mathcal{M}_{\eta,\beta}=\mathcal{M}_{\xi,\beta}$, for any $\epsilon<\alpha$ that satisfies $\epsilon=\pi(m,a_1,\ldots, a_n)$ $$\mathcal{M}_\eta\models Q_m(E_\eta^{-1}(a_1),E_\eta^{-1}(a_2),\ldots,E_\eta^{-1}(a_n))$$ if and only if $$\mathcal{M}_\xi\models Q_m(E_\xi^{-1}(a_1),E_\xi^{-1}(a_2),\ldots,E_\xi^{-1}(a_n)).$$
We conclude that $G$ is continuous.
\end{proof}

\begin{thm}
    Let $\kappa$ be a strongly inaccessible cardinal. If $\mathcal T_1$ is a theory with less than $\kappa$ non-isomorphic models of size $\kappa$ and $\mathcal T_2$ is unstable or superstable non-classifiable theory, then $\cong_{\mathcal T_1}\ \reduc\ \cong_{\mathcal T_2}$.
\end{thm}
\begin{proof}
    It follows from Proposition \ref{first_reduction} and Lemma \ref{red_cor}.
\end{proof}

Notice that for any stationary $S\subseteq \kappa$, $=^\mathbb F_{S_0}$ is bireducible with $=^\mathbb F_{S_0}\cap\ (\mathbb F^*\times \mathbb F^*)$. This is witness by the identity and the function:
$$\mathcal{F}(\eta)(\alpha)=\begin{cases} 
\eta(\alpha) &\mbox{if } \omega\leq\eta(\alpha)<\alpha\\
\eta(\alpha)+1 & \mbox{if } \eta(\alpha)<\omega\leq\alpha\\
0 & \mbox{if } \omega\leq\alpha\leq \eta(\alpha)\\
0 & \mbox{othewise. }
\end{cases}$$

\begin{prop}
    For all $f\in \mathbb F$ and $\lambda<\kappa$, $f=_{S_0}^\mathbb F g$ if and only if $T^f\cong T^g$.
\end{prop}

\begin{proof}
From Lemma \ref{representation_tree}, $f=_{S_0}^\mathbb F g$ implies $T^f\cong T^g$. 
For the other direction, let $\mathcal{T}$ be a countable complete unstable theory over a countable relational vocabulary. Let us suppose, towards contradiction, that $f\neq_{S_0}^\mathbb F g$ and $T^f\cong T^g$. Since $T^f\cong T^g$, $\mathcal{M}^f$ and $\mathcal{M}^g$ are isomorphic. From Lemma \ref{iso_direction}, we conclude that $f=_{S_0}^\mathbb F g$, a contradiction.
\end{proof}

\section*{Acknowledgements}
Due to extraordinary circumstances, the writing and editing of this research took longer than expected. 
Ido Feldman was supported by the Israel Science Foundation (grant agreement 203/22). 
Miguel Moreno was supported by the the Austrian Science Fund FWF, grant M 3210. 
Miguel Moreno was also supported by the European Research Council (ERC) under the European Union's Horizon 2020 research and innovation programme (grant agreement No 101020762).
Miguel Moreno was also supported by the Research Council of Finland (decision number 368671); and the project \textit{Perspectives on computational logic}, funded by the Research Council of Finland, project number 36942.

We thank Tapani Hyttinen for his feedback on Ehrenfeucht-Mostowski models, Assaf Rinot for his feedback on ordered coloured tree, Martin Hils for his feedback on forking sequences.

\providecommand{\bysame}{\leavevmode\hbox to3em{\hrulefill}\thinspace}
\providecommand{\MR}{\relax\ifhmode\unskip\space\fi MR }
\providecommand{\MRhref}[2]{%
  \href{http://www.ams.org/mathscinet-getitem?mr=#1}{#2} }
\providecommand{\href}[2]{#2}

\end{document}